\documentclass{amsproc}%
\usepackage{amsfonts}
\usepackage{amsmath}
\usepackage{amssymb}
\usepackage{graphicx}%
\providecommand{\U}[1]{\protect \rule{.1in}{.1in}}
\theoremstyle{plain}

\newtheorem{corollary}{Corollary}

\newtheorem{definition}{Definition}
\newtheorem{example}{Example}

\newtheorem{lemma}{Lemma}

\newtheorem{proposition}{Proposition}
\newtheorem{remark}{Remark}

\newtheorem{theorem}{Theorem}
\numberwithin{equation}{section}
\begin{document}
\title[Additive maps on $C\left(  X\right)  $]{Additive Local Multiplications and zero-preserving maps on $C\left(  X\right)
$}
\author{Qian Hu}
\address{East China University of Science and Technology, Shanghai, China}
\email{hu1267@126.com}
\dedicatory{Dedicated to Fuying Zhang, my primary school math teacher}
\begin{abstract}
Suppose $X$ is a compact Hausdorff space. In terms of topolocical properties
of $X$, we find topological conditions on $X$ that are equivalent to each of
the following: 1. every additive local multiplication on $C\left(  X\right)  $
is a multiplication, 2. every additive local multiplication on $C_{R}\left(
X\right)  $ is a multiplication, and 3. every additive map on $C\left(
X\right)  $ that is zero-preserving (i.e., $f\left(  x\right)  =0$ implies
$\left(  Tf\right)  \left(  x\right)  =0$) has the form $T\left(  f\right)
=T\left(  1\right)  \operatorname{Re}f+T\left(  i\right)  \operatorname{Im}f$.

\end{abstract}
\maketitle

\section{Introduction}

Suppose $X$ is a topological space. Let $C\left(  X\right)  $ and
$C_{R}\left(  X\right)  $ be the set of all complex continuous functions and
real continuous functions on $X$, respectively. This paper studies local
multiplications and zero-preserving maps on the algebra $C\left(  X\right)  $
and $C_{R}\left(  X\right)  $ when $X$ is a compact Hausdorff space. We find
an interesting interplay between the algebraic or linear-algebraic conditions
and unusual topological properties of the space $X$.

Suppose $\mathcal{A}$ is a ring with identity. A map $T$ on $\mathcal{A}$ is a
\textbf{local multiplication} if, for each $x\in$ $\mathcal{A}$, there is an
$a_{x}\in \mathcal{A}$ such that
\[
T\left(  x\right)  =a_{x}x\text{ (}xa_{x}\text{)}.
\]
The map $T$ is a \textbf{left} (\textbf{right}) \textbf{multiplication} if
there is an $a\in \mathcal{A}$ such that, for every $x\in \mathcal{A}$,
\[
T\left(  x\right)  =ax\text{ (}xa\text{).}%
\]
In this case we must have%
\[
a=T\left(  1\right)  \text{,}%
\]
and we write
\[
T=L_{T\left(  1\right)  }\text{ (}R_{T\left(  1\right)  }\text{)},
\]
i.e., left (right) multiplication by the element $T\left(  1\right)  $. When
that algebra $\mathcal{A}$ is commutative, there is no difference between left
and right, and we write $M_{T\left(  1\right)  }$ for $L_{T\left(  1\right)
}$. A map $T$ on $\mathcal{A}$ is an \textbf{additive map} if for every $x$
and $y$ in $\mathcal{A}$, $T\left(  x+y\right)  =T\left(  x\right)  +T\left(
y\right)  $.

There has been a lot of work characterizing cases in which every local
multiplication of a certain type is a multiplication. In 1983 the paper
\cite{H} of D. Hadwin contains an early result on local multiplications. In
1994 the so-called "Hadwin Lunch Bunch" \cite{HLB} gave necessary conditions
for local multiplications in rings with many idempotents to be multiplications.

In $1997$ D. Hadwin and J. W. Kerr \cite{HK} studied $\mathcal{R}$-linear
local multiplications on an algebra $\mathcal{A}$ over commutative ring
$\mathcal{R}$ with identity, and gave conditions that implied that every local
multiplication on $\mathcal{A}$ is a multiplication. When $\mathcal{R}$ is the
ring $\mathbb{Z}$ of integers, the $\mathcal{R}$-linear maps are simply the
additive maps. When $\mathcal{A}$ is a vector space over the rational numbers
$\mathbb{Q}$, the additive maps are precisely the $\mathbb{Q}$-linear maps.
When $\mathcal{A}$ is a topological vector space over $\mathbb{R}$, the
continuous additive maps are precisely the $\mathbb{R}$-linear ones. In
\cite{HK} D. Hadwin and J. W. Kerr proved, for a large class of unital
C*-algebras every additive local left (right) multiplication is a left (right)
multiplication. These algebras include ones for which the set of
finite-dimensional (dimension greater than $1$) separate the points of the
algebra. However, for additive maps, the commutative algebras, i.e., ones for
which every irreducible representation is $1$-dimensional, were not
considered. A unital C*-algebra is commutative if and only if it is isomorphic
to $C\left(  X\right)  $ for some compact Hausdorff space $X$. Note that
Hadwin and Kerr's results imply that every local left (right) multiplication
on the algebra $\mathbb{M}_{2}\left(  C\left(  X\right)  \right)  $ of
$2\times2$ matrices over $C\left(  X\right)  $ is a left (right) multiplication.

There is a vast literature on local derivations and local automorphisms, e.g.,
\cite{Chr}, \cite{Cr}, \cite{HL1}, \cite{HL2}, \cite{Han}, \cite{K}, \cite{S},
\cite{W}.

Another related active area of research is the study of maps that preserve a
particular property. According to MATHSCINET there have been almost $200$
papers studying linear preservers and at least $48$ papers studying additive
maps that preserve some property.

In this paper we restrict ourselves to the case where our algebra is the space
$C\left(  X\right)  $ of complex continuous functions on a compact Hausdorff
space $X$. It was shown in \cite[Theorem 6]{H} that every $\mathbb{C}$-linear
map on $C\left(  X\right)  $ that is a local multiplication is a
multiplication. We study additive or $\mathbb{R}$-linear local multiplications
on $C\left(  X\right)  $. We also study maps $T$ that are
\textbf{zero-preserving}, i.e., satisfying, for every $f\in C\left(  X\right)
$ and every $x\in X$,%
\[
f\left(  x\right)  =0\text{ implies }\left(  Tf\right)  \left(  x\right)  =0.
\]

In Section $2$ we consider those compact Hausdorff spaces $X$ for which every
additive local multiplication on $C\left(  X\right)  $ must be a
multiplication, and we call such $X$ an $\eta$\textbf{-space}. If every local
multiplication on $C_{R}\left(  X\right)  $ is a multiplication, we can $X$ a
\textbf{real} $\eta$\textbf{-space}. There is a vast difference between these
two concepts. We prove (Theorem \ref{slp}) that if the set of points $x\in X$
that are a limit of a \textbf{sequence} in $X\backslash \left \{  x\right \}  $
is dense in $X$, then $X$ is an $\eta$-space. In particular, if $X$ is first
countable, then $X$ is an $\eta$-space if and only if $X$ has no isolated
points. We also prove (Theorem \ref{union}) that the closure of the union of a
collection of $\eta$-subspaces of $X$ is also an $\eta$-space. Hence every
compact Hausdorff space has a unique maximal compact subspace that is an
$\eta$-space. It follows (Theorem \ref{product}) that the Cartesian product of
an $\eta$-space with any compact Hausdorff space is an $\eta$-space. We also
construct many spaces that are not $\eta$-spaces. The conjugation map on
$C\left(  X\right)  $ is $\mathbb{R}$-linear and is never a multiplication. We
prove (Theorem \ref{F}) that the conjugation map is a local multiplication on
$C\left(  X\right)  $ if and only if $X$ is an F-space in the sense of L.
Gillman and M. Jerison \cite{GJ}; this is also equivalent to the set of
$\mathbb{R}$-linear local multiplications on $C\left(  X\right)  $ being
precisely the maps of the form%
\[
T\left(  f\right)  =T\left(  1\right)  \operatorname{Re}\left(  f\right)
+T\left(  i\right)  \operatorname{Im}\left(  f\right)  \text{.}%
\]
We also prove (Theorem \ref{gfbar}) that there is a nonzero function $g$ such
that the map $T\left(  f\right)  =g\bar{f}$ is a local multiplication if and
only if there is a nonempty open F$_{\sigma}$-subset (i.e., a countable union
of closed sets) of $X$ that is an F-space. As a consequence, we prove
(Corollary \ref{rl}) that if no nonempty open F$_{\sigma}$-subset of $X$ is an
F-space, then every $\mathbb{R}$-linear (or continuous additive) local
multiplication on $C\left(  X\right)  $ is a multiplication. We also
characterize (Proposition \ref{betan}) the additive local multiplications on
$\ell^{\infty}=C\left(  \beta \left(  \mathbb{N}\right)  \right)  $, where
$\beta \left(  \mathbb{N}\right)  $ is the Stone-\v{C}ech compactification of
$\mathbb{N}$.

In Section $3$ we focus on the additive maps on $C\left(  X\right)  $ that are
zero-preserving, i.e., for every $f\in C\left(  X\right)  $ and for every
$x\in X$,%
\[
f\left(  x\right)  =0\Rightarrow T\left(  f\right)  \left(  x\right)
=0\text{.}%
\]
Since the conjugation map satisfies this property, we can't expect all of
these maps to be multiplications. But we can hope for them to be of the form%
\[
T\left(  f\right)  =T\left(  1\right)  \operatorname{Re}\left(  f\right)
+T\left(  i\right)  \operatorname{Im}\left(  f\right)  .
\]
We call a space $X$ for which every additive zero-preserving map has the above
form an $\upsilon$\textbf{-space}. We show (Theorem \ref{uCharacterization})
that this property is equivalent to every additive zero-preserving map on
$C_{R}\left(  X\right)  ,$ the set of real-valued continuous functions on $X$,
is a multiplication. Thus every $\upsilon$-space is a real $\eta$-space. This
characterization allows us to carry over results on $\eta$-spaces to those of
$\upsilon$-spaces. In particular we prove (Theorem \ref{fromsec2}) that, if
the set of sequential limit points is dense in $X$ or $X$ is the closure of
the union of a family of compact $\upsilon$-subspaces, or $X$ is the product
of a compact Hausdorff space and an $\upsilon$-space, then $X$ is an
$\upsilon$-space. This also shows that every compact Hausdorff space has a
unique maximal compact subspace that is an $\upsilon$-space. We also show
(Theorem \ref{zpAlg}) that every $\mathbb{R}$-linear zero-preserving map on
$C_{R}\left(  X\right)  $ is a multiplication.

In Section 4 We introduce the notions of $\mathfrak{q}$-point and strong
$\mathfrak{q}$-point, which are generalizations of a sequential limit point.
We prove (Theorem \ref{qp}) that if the set of strong $\mathfrak{q}$-points is
dense, then $X$ is an $\eta$-space, and if the set of $\mathfrak{q}$-points is
dense, then $X$ is an $\upsilon$-space. It turns out (Lemma \ref{pq}) that $x$
is a $q$-point if and only if $x$ is not a P-point in the sense of L. Gillman
and M. Henriksen \cite{GH}, and we show that the set of $\mathfrak{q}$-points
of $X$ is dense if and only if $X$ has no isolated points.

In Section 5, we present our main theorems. Our first main theorem (Theorem
\ref{main}) characterizes $\upsilon$-spaces and real $\eta$-spaces: Suppose
$X$ is a compact Hausdorff space. The following are equivalent:

\begin{enumerate}
\item $X$ is a $\upsilon$-space

\item $X$ is a real $\eta$-space

\item The set of $\mathfrak{q}$-points of $X$ is dense in $X$

\item $X$ has no isolated points.
\end{enumerate}

Our second main theorem (Theorem \ref{main2}) topologically characterizes
$\eta$-spaces: Suppose $X$ is a compact Hausdorff space. Then $X$ is an $\eta
$-space if and only if no nonempty open F$_{\sigma}$ set in $X$ is an F-space.

In Remark \ref{transfinite} we describe how to construct the maximal $\eta
$-subspace and the maximal $\upsilon$-subspace of a compact Hausdorff space
$X$.

We conclude with remarks about $\beta \left(  \mathbb{N}\right)  \backslash
\mathbb{N}$, where $\beta \left(  \mathbb{N}\right)  $ denotes the
Stone-\v{C}ech compactification of the set $\mathbb{N}$ of positive integers.
We have $\beta \left(  \mathbb{N}\right)  \backslash \mathbb{N}$ is an
$\upsilon$-space since $\beta \left(  \mathbb{N}\right)  \backslash \mathbb{N}$
has no isolated points. However, $\beta \left(  \mathbb{N}\right)
\backslash \mathbb{N}$ is not an $\eta$-space, since it is an F-space. We also
remark that W. Rudin \cite{R} and S. Shelah (see \cite{W}) have proved that
the assertion that every point in $\beta \left(  \mathbb{N}\right)
\backslash \mathbb{N}$ is a $q$-point is independent from the axioms of set
theory (ZFC).

For topological notions we refer the reader to \cite{Wil} and \cite{GJ}.

\section{$\eta$-spaces}

Suppose $Y$ is any completely regular Hausdorff space, then $\beta \left(
Y\right)  $ denotes the Stone-\v{C}ech Compactification of $Y$. For any real
number $r$, $[r]$ is the largest integer not more than $r$.

\begin{definition}
Suppose $X$ is a compact Hausdorff space. If every additive local
multiplication on $C\left(  X\right)  $ is a multiplication, then we call $X$
an $\mathbf{\eta}$\textbf{-space}. If every additive local multiplication on
$C_{R}\left(  X\right)  $ is a multiplication, we call $X$ a \textbf{real}
$\eta$\textbf{-space.}
\end{definition}

\begin{example}
Suppose $X=\left \{  a\right \}  $ is a singleton. Thus $C\left(  S\right)
\cong \mathbb{C}$. Every additive map ($\mathbb{Q}$-linear map) on $\mathbb{C}$
is a local multiplication. Given a linear basis $B$ for $\mathbb{C}$ over
$\mathbb{Q}$, we can get $card\left(  B\right)  =card\left(  \mathbb{R}%
\right)  $. Thus the cardinality of the set of all $\mathbb{Q}$-linear (i.e.,
additive) maps on $\mathbb{C}$ is $card\left(  B\right)  ^{card\left(
\mathbb{R}\right)  }=2^{\left(  2^{\aleph_{0}}\right)  }$. But the cardinality
of the set of all multiplications on $\mathbb{C}$ is $2^{\aleph_{0}}$. Thus no
singleton is an $\eta$-space.
\end{example}

\begin{definition}
Suppose $\{Y_{i}\}$ is a family of topology spaces. Let $Y$ be the disjoint
union of the $Y_{i}$'s (If there are two sets intersecting, then let $Y_{i}$
be $Y_{i}\times \left \{  i\right \}  $). Define a subset $U$ of $Y$ to be open
if and only if the intersection of $U$ and each $Y_{i}$ is open in $Y_{i}$.
Thus we defined the \textbf{disjoint union topology} on $Y$.
\end{definition}

\begin{theorem}
\label{direct}Suppose $X$ is the disjoint union of compact Hausdorff spaces
$Y$ and $Z$. Then $X$ is an $\eta$-space if and only if $Y$ and $Z$ are $\eta$-spaces.
\end{theorem}

\begin{proof}
We have $C\left(  X\right)  $ is isomorphic to the direct sum of $C\left(
Y\right)  $ and $C\left(  Z\right)  $. A map $T$ is a local multiplication on
$C\left(  X\right)  $ if and only if $T$ is the direct sum of local
multiplications on $C\left(  Y\right)  $ and $C\left(  Z\right)  $.
\end{proof}

\begin{corollary}
An $\eta$-space has no isolated points.
\end{corollary}

\begin{proof}
If an $\eta$-space $X$ has an isolated point $x$, then $X$ is the disjoint
union of $\left \{  x\right \}  $ and $X\backslash \left \{  x\right \}  $. Since
$\left \{  x\right \}  $ is not an $\eta$-space, from Theorem \ref{direct} it
follows that $X$ is not an $\eta$-space.
\end{proof}

\begin{definition}
Suppose $X$ is a topological space and $x\in X$. If there is a sequence
$\left \{  x_{n}\right \}  $ in $X\backslash \left \{  x\right \}  $ such that
$x_{n}\rightarrow x$, then we call $x$ a \textbf{sequential limit point} of
$X$.
\end{definition}

Note that if $X$ is a $T_{1}$ space, then $x\in X$ is a sequential limit point
if and only if there is a sequence$\left \{  x_{n}\right \}  $ in $X\backslash
\left \{  x\right \}  $ whose terms are different from each other such that
$x_{n}\rightarrow x$.

\begin{theorem}
\label{slp}Suppose $X$ is a compact Hausdorff space, and let $A$ be the set of
all sequential limit points of $X$. If $X=\bar{A}$, then $X$ is an $\eta$-space.
\end{theorem}

\begin{proof}
Suppose $T$ is an additive local multiplication on $C\left(  X\right)  $, $T$
is $\mathbb{Q}$-linear. Since the set of local multiplications on $C\left(
X\right)  $ is closed under linear combinations and compositions, $T$ is a
(local) multiplication if and only if\textbf{ }$T-M_{T\left(  1\right)  }$ is
a (local) multiplication. We may suppose $T\left(  1\right)  =0$, and prove
$T=0$.

First we prove $T\left(  a\cdot1\right)  =T\left(  a\right)  =0$ for every
$a\in \mathbb{R}$. Since $T$ is $\mathbb{Q}$-linear, we know, for every
$r\in \mathbb{Q}$, that
\[
T\left(  r\cdot1\right)  =rT\left(  1\right)  =0.
\]
Suppose $a\in \mathbb{R}$. Assume, via contradiction that $T\left(  a\right)
\left(  y\right)  \neq0$ for some $y\in X$. Since $A$ is dense in $X$, there
is an $x\in A$ such that $T\left(  a\right)  \left(  x\right)  \neq0$. Since
$x\in A$, there is a sequence $\left \{  x_{n}\right \}  $ in $X\backslash
\left \{  x\right \}  $ whose terms are different from each other such that
$x_{n}\rightarrow x$. Let $K=\left \{  x_{n}:n\in \mathbb{N}\right \}
\cup \left \{  x\right \}  $. Clearly, $K$ is a closed subset of $X$. Define a
function $f:K\rightarrow \lbrack0,1]$ by
\[
f\left(  x_{n}\right)  =a-\frac{[10^{n}a]}{10^{n}}\text{ for every }%
n\in \mathbb{N}\text{, and }f\left(  x\right)  =0\text{.}%
\]
Clearly, $f$ is continuous on $K$. Since $X$ is compact and Hausdorff, $X$ is
normal. By the Tietze extension theorem, there is a continuous function $F$
from $X$ to $[0,1]$ such that $F|_{K}=f$. Since $F\left(  x\right)  =0$ and
$T$ is a local multiplication, we know that $T\left(  F\right)  \left(
x\right)  =0.$ Thus
\[
T\left(  a-F\right)  \left(  x\right)  =T\left(  a\right)  \left(  x\right)
-T\left(  F\right)  \left(  x\right)
\]%
\[
=T\left(  a\right)  \left(  x\right)  -0=T\left(  a\right)  \left(  x\right)
\neq0\text{,}%
\]
i.e.,%
\[
T\left(  a-F\right)  \left(  x\right)  \neq0\text{.}%
\]
Since $T\left(  a-F\right)  \in C\left(  X\right)  $, and $x_{n}\rightarrow
x$, there exists a $k\in \mathbb{N}$ such that $T\left(  a-F\right)  \left(
x_{k}\right)  \neq0$. But%
\[
\left(  a-F\right)  \left(  x_{k}\right)  =a-f\left(  x_{k}\right)
=\frac{[10^{k}a]}{10^{k}}=r_{k}\in \mathbb{Q}\text{.}%
\]
Thus%
\[
\left(  a-F-r_{k}\right)  \left(  x_{k}\right)  =0.
\]%
\[
T\left(  a-F\right)  \left(  x_{k}\right)  =T\left(  a-F-r_{k}\right)  \left(
x_{k}\right)  =0\text{,}%
\]
since $T\left(  r_{k}\right)  =r_{k}\cdot T\left(  1\right)  =0$. This is a
contradiction. Thus for every $y\in X$, $T\left(  a\right)  \left(  y\right)
=0$, whence $T\left(  a\right)  =0$.

For every $g\in C_{R}\left(  X\right)  $ and every $x\in X$, $\left(
g-g\left(  x\right)  \cdot1\right)  \left(  x\right)  =0,$ and since $T$ is a
local multiplication,%

\[
0=T\left(  g-g\left(  x\right)  \cdot1\right)  \left(  x\right)  =T\left(
g\right)  \left(  x\right)  -T\left(  g\left(  x\right)  \cdot1\right)
=T\left(  g\right)  \left(  x\right)  \text{.}%
\]

Thus $T\left(  g\right)  =0$.

Now suppose $h\in C\left(  X\right)  $, and let $h=h_{1}+ih_{2}$, where
$h_{1}$, $h_{2}\in C_{R}\left(  X\right)  $. We have
\[
T\left(  h\right)  =T\left(  h_{1}\right)  +T\left(  ih_{2}\right)  =T\left(
ih_{2}\right)  =T\left(  i\right)  h_{2}\text{.}%
\]
(Set $L\left(  f\right)  =T\left(  if\right)  $ for every $f\in C\left(
X\right)  $. Thus $L$ is an additive local multiplication on $C\left(
X\right)  $. From above, if $L\left(  1\right)  =0$, then $L\left(  g\right)
=0$ for every $g\in C_{R}\left(  X\right)  $. Thus $L\left(  g\right)
=L\left(  1\right)  g$, i.e., $T\left(  ig\right)  =T\left(  i\right)  g$ for
any $g\in C_{R}\left(  X\right)  $.) To prove $T\left(  h\right)  =0$, it's
enough to show $T\left(  i\right)  =0$.

For every $x\in A$, there is a sequence$\left \{  x_{n}\right \}  $ in
$X\backslash \left \{  x\right \}  $ whose terms are different from each other
such that $x_{n}\rightarrow x$. Let $K=\left \{  x_{n}:n\in \mathbb{N}\right \}
\cup \left \{  x\right \}  $. Then $K$ is a closed subset of $X$. Define a
function $f:K\rightarrow \lbrack0,1]$ by
\[
f\left(  x_{n}\right)  =\frac{1}{n}\cdot|\sin \left(  \frac{n\pi}{2}\right)
|\text{ for every }n\in \mathbb{N}\text{, and }f\left(  x\right)  =0\text{.}%
\]
Thus $f$ is continuous on $K$. By the Tietze extension theorem, there is a
continuous function $F$ from $X$ to $[0,1]$ such that $F|_{K}=f$. Similarly,
define $g:K\rightarrow \lbrack0,1]$ by
\[
g\left(  x_{n}\right)  =\frac{1}{n}\text{ for every }n\in \mathbb{N}\text{, and
}g\left(  x\right)  =0\text{.}%
\]
We get a continuous function $G$ from $X$ to $[0,1]$ such that $G|_{K}=g$. Set
$H=F+iG\in C\left(  X\right)  $. Thus
\[
T\left(  H\right)  =T\left(  F+iG\right)  =T\left(  i\right)  G=h_{H}%
\cdot \left(  F+iG\right)
\]
for some $h_{H}\in C\left(  X\right)  $. Thus
\begin{align*}
T\left(  i\right)  \left(  x_{n}\right)  G\left(  x_{n}\right)   &
=h_{H}\left(  x_{n}\right)  \cdot \left(  F\left(  x_{n}\right)  +iG\left(
x_{n}\right)  \right) \\
T\left(  i\right)  \left(  x_{n}\right)  \cdot \frac{1}{n}  &  =h_{H}\left(
x_{n}\right)  \cdot \left(  \frac{1}{n}\cdot|\sin \left(  \frac{n\pi}{2}\right)
|+i\frac{1}{n}\right) \\
T\left(  i\right)  \left(  x_{n}\right)   &  =h_{H}\left(  x_{n}\right)
\cdot|\sin \left(  \frac{n\pi}{2}\right)  |+ih_{H}\left(  x_{n}\right)
\end{align*}

for every $n\in \mathbb{N}$.

Clearly, we have $T\left(  i\right)  \left(  x_{n}\right)  =ih_{H}\left(
x_{n}\right)  $ when $n$ is even, and $T\left(  i\right)  \left(
x_{n}\right)  =h_{H}\left(  x_{n}\right)  +ih_{H}\left(  x_{n}\right)  $ when
$n$ is odd. Since $T\left(  i\right)  \left(  x_{n}\right)  \rightarrow
T\left(  i\right)  \left(  x\right)  $ as $n\rightarrow \infty$, we have
\[
T\left(  i\right)  \left(  x\right)  =ih_{H}\left(  x\right)  =h_{H}\left(
x\right)  +ih_{H}\left(  x\right)  \text{.}%
\]
Thus $h_{H}\left(  x\right)  =0$. Thus $T\left(  i\right)  \left(  x\right)
=0$ for every $x\in A$. Since $X=\bar{A}$, $T\left(  i\right)  =0$. Thus
$T\left(  h\right)  =0$ for every $h\in C\left(  X\right)  $, i.e., $T=0$.
Thus every additive local multiplication on $C\left(  X\right)  $ is a
multiplication, i.e., $X$ is an $\eta$-space\textbf{.}
\end{proof}

\begin{remark}
The preceding theorem tells us whether $X$ is an $\eta$-space has nothing to
do with the connectedness of $X$. From the theorem, we know the Cantor set and
the closed interval $[0,1]$ are $\eta$-spaces but the former is totally
disconnected and the latter is connected. Note that in this theorem $X=\bar
{A}$ implies has no isolated points.
\end{remark}

\begin{corollary}
Suppose $X$ is a compact Hausdorff space. If $X$ is first countable, then $X$
is an $\eta$-space if and only if $X$ has no isolated points.
\end{corollary}

\begin{corollary}
Suppose $Y$ is a completely regular Hausdorff space, and let $A$ be the set of
all sequential limit points of $Y$. If $A$ is dense in $Y$, then $\beta \left(
Y\right)  $ is an $\eta$-space.
\end{corollary}

\bigskip

Next theorem is another version of \cite[Theorem 6]{HK}.

\begin{theorem}
\label{union}Suppose $X$ is a compact Hausdorff space, and $\left \{
K_{i}\right \}  $ is a collection of closed subset of $X$. If each $K_{i}$ is
an $\eta$-space, then the closure of the union of $K_{i}$'s is an $\eta$-space.
\end{theorem}

\begin{proof}
Let $K$ be the closure of the union of $K_{i}$'s. Suppose $T$ is an additive
local multiplication on $C\left(  K\right)  $ and $T\left(  1\right)  =0$. For
each $i$ in $I$, define $T_{i}$ on $C\left(  K_{i}\right)  $ by
\[
T_{i}\left(  f\right)  =T\left(  \widetilde{f}\right)  |_{K_{i}}\text{ for
every }f\in C\left(  K_{i}\right)  \text{,}%
\]

where $\widetilde{f}\in C\left(  K\right)  $ is a Tietze extension of $f$. The
definition is well-defined. In fact, suppose $\widetilde{f_{1}}$,
$\widetilde{f_{2}}$ $\in C\left(  K\right)  $ and $\widetilde{f_{1}}|_{K_{i}%
}=$ $\widetilde{f_{2}}|_{K_{i}}$. Thus%
\[
T\left(  \widetilde{f_{1}}\right)  -T\left(  \widetilde{f_{2}}\right)
=T\left(  \widetilde{f_{1}}-\widetilde{f_{2}}\right)  =h\cdot \left(
\widetilde{f_{1}}-\widetilde{f_{2}}\right)
\]

for some $h\in C\left(  K\right)  $. Then%
\begin{align*}
T\left(  \widetilde{f_{1}}\right)  |_{K_{i}}-T\left(  \widetilde{f_{2}%
}\right)  |_{K_{i}}  &  =T\left(  \widetilde{f_{1}}-\widetilde{f_{2}}\right)
|_{K_{i}}=h|_{K_{i}}\cdot \left(  \widetilde{f_{1}}-\widetilde{f_{2}}\right)
|_{K_{i}}\\
&  =0\text{.}%
\end{align*}

Thus $T\left(  \widetilde{f_{1}}\right)  |_{K_{i}}=T\left(  \widetilde{f_{2}%
}\right)  |_{K_{i}}$. Clearly, $T_{i}$ is an additive local multiplication on
$C\left(  K_{i}\right)  $. Since $K_{i}$ is an $\eta$-space and $T_{i}\left(
1\right)  =0$, we have $T_{i}=0$. Thus for every $g\in C\left(  K\right)  $,
$T\left(  g\right)  |_{K_{i}}=0$ for each $i$ in $I$. Since $K$ is the closure
of the union of $K_{i}$'s, $T\left(  g\right)  =0$. Since $g$ is arbitrary, we
have $T=0$. Thus $K$ is an $\eta$-space.
\end{proof}

\begin{theorem}
\label{product}Suppose $Y$ is a compact Hausdorff space, and $K$ is an $\eta
$-space. Then $Y\times K$ is an $\eta$-space.
\end{theorem}

\begin{proof}
For any $y\in Y$, $\left \{  y\right \}  \times K$ is a closed $\eta$-subspace
of $Y\times K$. Thus from the last theorem, $Y\times K=\cup_{y\in Y}\left \{
y\right \}  \times K$ is an $\eta$-space.
\end{proof}

\begin{corollary}
Every compact Hausdorff space is homeomorphic to a subspace of an $\eta$-space.
\end{corollary}

\begin{theorem}
Every compact Hausdorff space has a unique maximal $\eta$-subspace.
\end{theorem}

\begin{proof}
Suppose $X$ is a compact Hausdorff space, and let $K$ be the closure of the
union of all $\eta$-subspaces of $X$. Thus $K$ is an $\eta$-space. Any $\eta
$-subspace of $X$ is a subset of $K$. Thus $K$ is the unique maximal $\eta
$-subspace of $X$.
\end{proof}

\begin{definition}
Suppose $X$ is a completely regular Hausdorff space. For every $f\in C\left(
X\right)  $, let $Z\left(  f\right)  =\left \{  x\in X:f\left(  x\right)
=0\right \}  $ be the\textbf{ zero-set of }$f$\textbf{.} Any set that is a
zero-set of some function in $C\left(  X\right)  $ is called a
\textbf{zero-set in }$X$\textbf{.} We call the complement of a zero-set a
\textbf{cozero-set}. We say a subspace $S$ of $X$ is $R$\textbf{-embedded} if
every bounded function in $C_{R}\left(  S\right)  $ can be extended to a
bounded function in $C_{R}\left(  X\right)  $. If every cozero-set in $X$ is
$R$-embedded, we call $X$ an \textbf{F-space }\cite[Theorem 14.25(6)]{GJ}.
\end{definition}

If $g\in C\left(  X\right)  $ is nonzero, then the map $T\left(  f\right)
=g\bar{f}$ is not a multiplication. We now characterize the spaces $X$ for
which $T$ is a local multiplication.

For any $a$, $b\in \mathbb{R}$, the symbol $a\vee b$ denotes $\sup \left \{
a\text{, }b\right \}  $. Likewise, $a\wedge b$ denotes $\inf \left \{  a\text{,
}b\right \}  $. For any $f$, $g\in C_{R}\left(  X\right)  $, define $\left(
f\vee g\right)  \left(  x\right)  =f\left(  x\right)  \vee g\left(  x\right)
$, for every $x\in X$. Thus $f\vee g\in C_{R}\left(  X\right)  $. Dually, we
defined $f\wedge g\in C_{R}\left(  X\right)  $. In the proof of the next
theorem, we use an equivalent condition for $X$ to be an F-space \cite[Theorem
14.25(5)]{GJ}, i.e., given $f\in C_{R}\left(  X\right)  $, there exists $k\in
C_{R}\left(  X\right)  $ such that $f=k\cdot \left \vert f\right \vert $.

\begin{lemma}
\label{oldgfbar}Suppose $X$ is a compact Hausdorff space, and $0\neq g\in
C\left(  X\right)  $. Let $A=X\backslash Z\left(  g\right)  $. Define
$T\left(  f\right)  =g\cdot \bar{f}$ for every $f\in C\left(  X\right)  $. Then
$T$ is a local multiplication if and only if $A$ is an F-space.
\end{lemma}

\begin{proof}
$\Leftarrow$: Suppose $A$ is an F-space. For every $f\in C\left(  X\right)  $,
$f|_{A}\in C\left(  A\right)  $. Denote $f|_{A}$ by $f_{A}$. Set $f_{A}=u+iv$,
where $u$, $v\in C_{R}\left(  A\right)  $. Thus on $A\backslash Z\left(
f_{A}\right)  $, we have%
\[
\frac{\bar{f}_{A}}{f_{A}}=\frac{u^{2}-v^{2}}{u^{2}+v^{2}}-i\frac{2uv}%
{u^{2}+v^{2}}\text{.}%
\]

Let $h_{1}=\frac{u^{2}-v^{2}}{u^{2}+v^{2}}$ and $h_{2}=\frac{2uv}{u^{2}+v^{2}%
}$. Clearly, $h_{1}$, $h_{2}$ are bounded real continuous functions on
$A\backslash Z\left(  f_{A}\right)  =A\backslash Z\left(  u^{2}+v^{2}\right)
$, where $u^{2}+v^{2}\in C_{R}\left(  A\right)  $. Since $A$ is an F-space,
$h_{1}$, $h_{2}$ have bounded continuous extensions on $A$, say $\widetilde
{h_{1}}$, $\widetilde{h_{2}}$. Define%
\[
h=g\cdot \left(  \widetilde{h_{1}}-i\widetilde{h_{2}}\right)  \text{ on
}A\text{, and }h=0\text{ on }Z\left(  g\right)  \text{.}%
\]

Thus $h\in C\left(  X\right)  $ and $h\cdot f=g\cdot \bar{f}$.

Thus $T\left(  f\right)  =g\cdot \bar{f}$ is a local multiplication.

$\Longrightarrow$: Suppose $T$ is a local multiplication. For each $h\in
C_{R}\left(  A\right)  $, if $h$ is not bounded, choose an $r>0$ and let
$h_{r}=\left(  -r\vee h\right)  \wedge r$, then $h_{r}$ is a bounded real
continuous function on $A$. Define%
\[
k=g\cdot \left(  h_{r}+i|h_{r}|\right)  \text{ on }A\text{, and }k=0\text{ on
}Z\left(  g\right)  \text{.}%
\]

Thus $k\in C\left(  X\right)  $. Since $T$ is a local multiplication, there is
an $h_{k}\in C\left(  X\right)  $ such that $T\left(  k\right)  =g\cdot \bar
{k}=h_{k}\cdot k$. Thus on $A\backslash Z\left(  h_{r}\right)  $, we have%
\[
h_{k}=g\cdot \frac{\bar{k}}{k}=g\cdot \frac{\bar{g}\cdot \left(  h_{r}%
-i|h_{r}|\right)  }{g\cdot \left(  h_{r}+i|h_{r}|\right)  }=-i\bar{g}\cdot
\frac{h_{r}}{|h_{r}|}\text{,}%
\]

i.e., $h_{k}=-i\bar{g}\cdot \frac{h_{r}}{|h_{r}|}$. Since $\bar{g}\neq0$ on
$A$, we have%
\[
\frac{ih_{k}}{\bar{g}}=\frac{h_{r}}{|h_{r}|}%
\]

on $A\backslash Z\left(  h_{r}\right)  $. Set $\frac{ih_{k}}{\bar{g}}%
=f_{1}+if_{2}$, where $f_{1}$, $f_{2}\in C_{R}\left(  A\right)  $. From above,
$f_{1}=\frac{h_{r}}{|h_{r}|}$, and $f_{2}=0$ on $A\backslash Z\left(
h_{r}\right)  $. Thus $h_{r}=f_{1}\cdot|h_{r}|$ on $A$.

We say $h=f_{1}\cdot|h|$ on $A$. Since if $h\left(  x\right)  <-r$ for some
$x\in A$, we have $h_{r}\left(  x\right)  =-r$. Then $f_{1}\left(  x\right)
=-1$ and $h\left(  x\right)  =-|h\left(  x\right)  |=f_{1}\left(  x\right)
\cdot|h\left(  x\right)  |$. For $h\left(  x\right)  >r$, we get the same
result. Thus we proved that for every $h\in C_{R}\left(  A\right)  $, there
exists an $f_{1}\in C_{R}\left(  A\right)  $ such that $h=f_{1}\cdot|h|$,
i.e., $A$ is an F-space.
\end{proof}

\bigskip

\begin{theorem}
\label{F}Suppose $X$ is a compact Hausdorff space. The following are equivalent:

\begin{enumerate}
\item The conjugation map is a local multiplication on $C\left(  X\right)  $.

\item $X$ is an F-space.

\item The set of $\mathbb{R}$-linear local multiplications on $C\left(
X\right)  $ consists of all maps of the form%
\[
T\left(  f\right)  =T\left(  1\right)  \operatorname{Re}\left(  f\right)
+T\left(  i\right)  \operatorname{Im}\left(  f\right)  \text{.}%
\]

\end{enumerate}
\end{theorem}

\begin{proof}
The equivalent of $\left(  1\right)  $ and $\left(  2\right)  $ is immediately
derived from Lemma \ref{oldgfbar}.

$\left(  2\right)  \Leftrightarrow \left(  3\right)  $: Suppose $T$ is an
$\mathbb{R}$-linear local multiplication on $C\left(  X\right)  $. It is clear
$T$ has the form $T\left(  f\right)  =T\left(  1\right)  \operatorname{Re}%
\left(  f\right)  +T\left(  i\right)  \operatorname{Im}\left(  f\right)  $ for
any $f\in C\left(  X\right)  $. Conversely, if $T$ is a map on $C\left(
X\right)  $ which has the form $T\left(  f\right)  =T\left(  1\right)
\operatorname{Re}\left(  f\right)  +T\left(  i\right)  \operatorname{Im}%
\left(  f\right)  $, then of course it is $\mathbb{R}$-linear. Also we have%
\[
T\left(  f\right)  =T\left(  1\right)  \operatorname{Re}\left(  f\right)
+T\left(  i\right)  \operatorname{Im}\left(  f\right)  =\left[  \frac{T\left(
1\right)  -iT\left(  i\right)  }{2}\right]  \cdot f+\left[  \frac{T\left(
1\right)  +iT\left(  i\right)  }{2}\right]  \cdot \bar{f}\text{.}%
\]
Thus $T$ is a local multiplication if and only if $\left[  \frac{T\left(
1\right)  +iT\left(  i\right)  }{2}\right]  \cdot \bar{f}$ is a local
multiplication. If $X$ is an F-space, then every cozero-set in $X$ is an
F-space \cite[14.26]{GJ}. From Lemma \ref{oldgfbar}, we have $\left[
\frac{T\left(  1\right)  +iT\left(  i\right)  }{2}\right]  \cdot \bar{f}$ is a
local multiplication, i.e., $T$ is a local multiplication. Thus we prove
$\left(  2\right)  \Rightarrow \left(  3\right)  $.

If every $T$ of the form $T\left(  f\right)  =T\left(  1\right)
\operatorname{Re}\left(  f\right)  +T\left(  i\right)  \operatorname{Im}%
\left(  f\right)  $ is a local multiplication, then $\left[  \frac{T\left(
1\right)  +iT\left(  i\right)  }{2}\right]  \cdot \bar{f}$ is a local
multiplication. From Lemma \ref{oldgfbar}, $X\backslash Z\left(
\frac{T\left(  1\right)  +iT\left(  i\right)  }{2}\right)  $ is an F-space.
Choose $T\left(  1\right)  $ and $T\left(  i\right)  $ such that $Z\left(
\frac{T\left(  1\right)  +iT\left(  i\right)  }{2}\right)  $ is empty. Then
$X$ is an F-space.
\end{proof}

The following theorem is an immediate consequence of Lemma \ref{oldgfbar} and
the fact that $A\subseteq X$ is an open F$_{\sigma}$ if and only if there is a
$g\in C\left(  X\right)  $ such that $A=X\backslash Z\left(  g\right)  $ (See
\cite{GJ}.)

\begin{theorem}
\label{gfbar}Suppose $X$ is a compact Hausdorff space. The following are equivalent:

\begin{enumerate}
\item There is a nonzero function $g$ $\in C\left(  X\right)  $ such that the
map $T\left(  f\right)  =g\bar{f}$ is a local multiplication on $C\left(
X\right)  $.

\item There is a nonempty open F$_{\sigma}$ subset of $X$ that is an F-space.
\end{enumerate}
\end{theorem}

\begin{corollary}
\label{rl}No nonempty open F$_{\sigma}$-subset of $X$ is an $F$-space if and
only if every $\mathbb{R}$-linear local multiplication on $X$ is a multiplication.
\end{corollary}

\begin{proof}
The sufficiency follows from Theorem \ref{gfbar}, since a map $T\left(
f\right)  =g\bar{f}$ with $g\neq0$ is never a multiplication. Conversely,
suppose no nonempty open F$_{\sigma}$-subset of $X$ is an $F$-space. Suppose
$T$ is a real linear local multiplication on $C\left(  X\right)  $. We may
suppose $T\left(  1\right)  =0$. Thus $T\left(  u\right)  =0$ for every $u\in
C_{R}\left(  X\right)  $. Thus $T\left(  f\right)  =T\left(  i\right)
\operatorname{Im}f$ for any $f$, $g\in C_{R}\left(  X\right)  $. If $T\left(
i\right)  \neq0$, we have%
\[
\frac{i}{2}T\left(  i\right)  \bar{f}=T\left(  i\right)  \left(
\operatorname{Im}f\right)  +M_{\frac{i}{2}T\left(  i\right)  }f
\]
is a local multiplication. It follows from Theorem \ref{gfbar} that
$X\backslash Z\left(  T\left(  i\right)  \right)  $ is a nonempty F-space,
which is a contradiction.

Thus $T\left(  i\right)  =0$ and $T=0$, i.e., $T$ is a multiplication.
\end{proof}

\begin{example}
If $X$ is an F-space, then the maximal $\eta$-subspace is the empty set. If
$X$ is the union of a compact $\eta$-space $A$ and a compact F-space $Y$, then
$A$ is the maximal $\eta$-subspace of $X$. To see this, suppose $A\subseteq K$
and $A\neq K$ and $K$ is compact. Choose $b\in K\backslash A$ and choose a
continuous function $g:X\rightarrow \left[  0,1\right]  $ such that $g|_{A}=0$
and $g\left(  b\right)  =1$. Suppose $f\in C\left(  X\right)  $. Since $Y$ is
an $F$-space, we know from Theorem \ref{F} and the Tietze extension theorem
that there is an $h\in C\left(  X\right)  $ such that $\bar{f}=hf$ on $Y.$
Since $g|A=0,$ we see that
\[
g\bar{f}=ghf
\]
on $X.$ Thus $T\left(  f\right)  =g\bar{f}$ is a local multiplication that is
not a multiplication. Thus $T\left(  f|K\right)  =\left(  g|_{K}\right)
\left(  \bar{f}|_{K}\right)  $ defines a local multiplication on $C\left(
K\right)  $ that is not a multiplication.
\end{example}

If $Y$ is a completely regular Hausdorff space, let $C_{b}\left(  Y\right)  $
denote the bounded continuous functions on $Y$. Then $C_{b}\left(  Y\right)
\cong C\left(  \beta \left(  Y\right)  \right)  $. We can say $Y$ is an $\eta
$-space when $\beta \left(  Y\right)  $ is an $\eta$-space.

\begin{theorem}
\label{lico}Suppose $I$ is an infinite set. Suppose $Y_{i}$ is a nonempty
compact $\eta$-space for each $i$ in $I$, and $Y$ is the disjoint union of
$Y_{i}$'s. Let $K=\beta \left(  Y\right)  \backslash Y$. Then $K$ is an F-space
if and only if $I$ is countable.
\end{theorem}

\begin{proof}
A function $f$ in $C_{b}\left(  Y\right)  $ corresponds to a uniformly bounded
family $\left \{  f_{i}\right \}  $ with $f_{i}$ in $C\left(  Y_{i}\right)  $.
The function $f$ is $0$ on $K$ if and only if for every $r>0$, there is a
finite subset $D$ of $I$ so that $|f_{i}|<r$ whenever $i$ is not in $D$. Let
$J$ be the set of all functions in $C_{b}\left(  Y\right)  $ which vanish on
$K$. Thus $C\left(  K\right)  \cong C_{b}\left(  Y\right)  /J$.

$\Longrightarrow$: Suppose $K$ is an F-space, thus the conjugation map is a
local multiplication on $C\left(  K\right)  \cong C_{b}\left(  Y\right)  /J$.
For every $f\in C_{b}\left(  Y\right)  $, we have $\bar{f}-h_{f}\cdot f\in J$
for some $h_{f}\in C_{b}\left(  Y\right)  $. Thus for every $n\in \mathbb{N}$,
there is a finite subset $D_{\left(  f,n\right)  }$ of $I$ so that $|\bar
{f}_{i}-h_{f_{i}}\cdot f_{i}|<\frac{1}{n}$ whenever $i$ is not in $D_{\left(
f,n\right)  }$. Set $D_{f}=%
{\textstyle \bigcup}
D_{\left(  f,n\right)  }$. Thus $D_{f}$ is countable.

Since each $Y_{i}$ is a compact $\eta$-space, the conjugation map is not a
local multiplication on $C\left(  Y_{i}\right)  $. Thus there is a $g_{i}\in
C\left(  Y_{i}\right)  $ such that $\bar{g}_{i}-h\cdot g_{i}\neq0$ for every
$h\in C\left(  Y_{i}\right)  $. Let $g=%
{\textstyle \bigcup}
g_{i}\in C\left(  Y\right)  $. For every $n\in \mathbb{N}$, set
\[
g\left(  n\right)  =\left(  -n\vee \operatorname{Re}g\right)  \wedge n+i\left(
-n\vee \operatorname{Im}g\right)  \wedge n\text{.}%
\]
Thus $g\left(  n\right)  \in C_{b}\left(  Y\right)  $ for every $n\in
\mathbb{N}$. Set $D=%
{\textstyle \bigcup}
D_{g\left(  n\right)  }$ which $D_{g\left(  n\right)  }$ is defined similarly
with $D_{f}$ for every $n\in \mathbb{N}$. Since each $D_{g\left(  n\right)  }$
is countable, we have $D$ is countable. Also $D\subseteq I$.

If there exists an $m\in I$ and $m\notin D$, we have $g_{m}\in C\left(
Y_{m}\right)  =C_{b}\left(  Y_{m}\right)  $. There is a $k\in \mathbb{N}$ so
that $g_{m}=g\left(  k\right)  _{m}$. Since $g\left(  k\right)  \in
C_{b}\left(  Y\right)  $, we have $\overline{g\left(  k\right)  }-h\cdot
g\left(  k\right)  \in J$ for some $h\in C_{b}\left(  Y\right)  $. Thus for
every $n\in \mathbb{N}$, we have $|\overline{g\left(  k\right)  }_{m}%
-h_{m}\cdot g\left(  k\right)  _{m}|<\frac{1}{n}$, since $m\notin D\supseteq
D_{g\left(  k\right)  }$. Thus $|\overline{g\left(  k\right)  }_{m}-h_{m}\cdot
g\left(  k\right)  _{m}|=0$. Since $g_{m}=g\left(  k\right)  _{m}$, we have
$\bar{g}_{m}=h_{m}\cdot g_{m}$ where $h_{m}\in C\left(  Y_{m}\right)  $. This
contradicts our choice of $g_{m}$. Thus $I=D$, i.e., we get $I$ is countable.

$\Leftarrow$: Suppose $I$ is countable. Let $I=\mathbb{N}$. For every $f\in
C_{b}\left(  Y\right)  $, $f_{n}\in C\left(  Y_{n}\right)  $ for every
$n\in \mathbb{N}$. Let $F_{n}=\left \{  x\in Y_{n}\text{: }|f_{n}\left(
x\right)  |\geq \frac{1}{n}\right \}  $, and $F_{n}$ is a closed subset of
$Y_{n}$. Thus $\frac{\bar{f}_{n}}{f_{n}}$ is continuous on $F_{n}$ and
$|\frac{\bar{f}_{n}}{f_{n}}|=1$. By the Tietze extension theorem, there is a
continuous function $h_{n}$ from $Y_{n}$ to $\mathbb{C}$ such that.
$h_{n}|_{F_{n}}=\frac{\bar{f}_{n}}{f_{n}}$ and $|h_{n}|\leq2$. Thus on
$Y_{n}\backslash F_{n}$, we have
\[
|\bar{f}_{n}-h_{n}\cdot f_{n}|\leq|\bar{f}_{n}|+|h_{n}|\cdot|f_{n}|<\frac
{1}{n}+2\cdot \frac{1}{n}=\frac{3}{n}\text{ .}%
\]
On $F_{n}$, we have $|\bar{f}_{n}-h_{n}\cdot f_{n}|=0$. Thus $|\bar{f}%
_{n}-h_{n}\cdot f_{n}|<\frac{3}{n}$ on $Y_{n}$.

Set $h=%
{\textstyle \bigcup}
h_{n}$. Since $|h_{n}|<2$ for every $n\in \mathbb{N}$, we have $h\in
C_{b}\left(  Y\right)  $. Thus $\bar{f}-h\cdot f\in J$. Thus the conjugation
map is a local multiplication on $C_{b}\left(  Y\right)  /J\cong C\left(
K\right)  $, i.e., we get $K$ is an F-space.
\end{proof}

\begin{remark}
Suppose $X$ is a compact $\eta$-space. We know that every clopen subset of $X$
is an $\eta$-space. From Theorem \ref{lico}, we get a closed subset
$K=\beta \left(  Y\right)  \backslash Y$ of a compact $\eta$-space
$\beta \left(  Y\right)  $ which is not an $\eta$-space, in fact, an F-space.
\end{remark}

\begin{corollary}
Suppose $I$ is uncountable and $Y$ is the disjoint union of nonempty compact
Hausdorff spaces $Y_{i}$'s with $i$ in $I$. Let $K=\beta \left(  Y\right)
\backslash Y$. Then $K$ is an F-space if and only if all but countable many
$Y_{i}$'s are F-spaces.
\end{corollary}

Let $\mathbb{N}$ be all positive integers with discrete topology. Since the
conjugation map is a local multiplication on $C_{b}\left(  \mathbb{N}\right)
\cong C\left(  \beta \left(  \mathbb{N}\right)  \right)  $, we get that
$\beta \left(  \mathbb{N}\right)  $ is an F-space. We have characterized all
real linear local multiplications on $C\left(  \beta \left(  \mathbb{N}\right)
\right)  $. Next we will see how an additive local multiplication on $C\left(
\beta \left(  \mathbb{N}\right)  \right)  $ relates to a real linear local multiplication.

\begin{proposition}
\label{betan}Suppose $T$ is an additive local multiplication on $C\left(
\beta \left(  \mathbb{N}\right)  \right)  \cong C_{b}\left(  \mathbb{N}\right)
$. Then $T\left(  f\right)  =T\left(  1\right)  \cdot f$ for every $f\in
C_{R}\left(  \beta \left(  \mathbb{N}\right)  \right)  $ except finite many
points of $\mathbb{N}$.
\end{proposition}

\begin{proof}
Clearly, $C_{b}\left(  \mathbb{N}\right)  $ is the set of all bounded complex
sequences. Let $C_{0}\left(  \mathbb{N}\right)  $ be the set of all sequences
which converge to $0$. We may suppose $T\left(  1\right)  =0$. Thus $T\left(
r\cdot1\right)  =T\left(  r\right)  =0$ for every $r\in \mathbb{Q}$. Also for
every bounded rational sequences $\left \{  r_{n}\right \}  $, $T\left(
\left \{  r_{n}\right \}  \right)  =0$. Since for every $n\in \mathbb{N}$,
$T\left(  \left \{  r_{n}\right \}  \right)  \left(  n\right)  =T\left(
\left \{  r_{n}\right \}  -r_{n}\cdot1\right)  \left(  n\right)  =h\left(
n\right)  \cdot \left(  r_{n}-r_{n}\right)  =0$ for some $h\in C_{b}\left(
\mathbb{N}\right)  $. For every $a\in \mathbb{R}$, there is a sequence
$\left \{  a_{n}\right \}  $ in $\mathbb{Q}$ such that. $a_{n}\rightarrow a$.
Thus%
\[
T\left(  \left \{  a_{n}\right \}  -a\right)  =-T\left(  a\right)
=h\cdot \left(  \left \{  a_{n}\right \}  -a\right)
\]

for some $h\in C_{b}\left(  \mathbb{N}\right)  $. Thus for every $n\in%
\mathbb{N}
$, $-T\left(  a\right)  \left(  n\right)  =h\left(  n\right)  \cdot \left(
a_{n}-a\right)  $. Let $n\rightarrow \infty$,%
\[
|T\left(  a\right)  \left(  n\right)  |=|h\left(  n\right)  |\cdot
|a_{n}-a|\rightarrow0\text{,}%
\]

since $h$ is a bounded sequence. Thus $T\left(  a\right)  \in C_{0}\left(
\mathbb{N}\right)  $ for every $a\in \mathbb{R}$.

Let $f$ be any bounded real sequence. It has a convergent subsequence
$\left \{  f\left(  n_{k}\right)  \right \}  $ and $f\left(  n_{k}\right)
\rightarrow b$ as $k\rightarrow \infty$ for some $b\in \mathbb{R}$. Thus%
\[
T\left(  f-b\right)  =T\left(  f\right)  -T\left(  b\right)  =h_{b}%
\cdot \left(  f-b\right)
\]

for some $h_{b}\in C_{b}\left(  \mathbb{N}\right)  $. For every $k\in
\mathbb{N}$,%
\[
T\left(  f\right)  \left(  n_{k}\right)  -T\left(  b\right)  \left(
n_{k}\right)  =h_{b}\left(  n_{k}\right)  \cdot \left(  f\left(  n_{k}\right)
-b\right)  \text{.}%
\]

Let $k\rightarrow \infty$, we get $T\left(  f\right)  \left(  n_{k}\right)
\rightarrow0$.

Thus every subsequence $\left \{  T\left(  f\right)  \left(  n_{k}\right)
\right \}  $ of $T\left(  f\right)  $ has a subsequence $\left \{  T\left(
f\right)  \left(  n_{k_{j}}\right)  \right \}  $ that converges to $0$. Then
$T\left(  f\right)  \in C_{0}\left(  \mathbb{N}\right)  $ for every bounded
$f\in C_{R}\left(  \mathbb{N}\right)  $. Thus $T\left(  f\right)  $ is $0$ on
$\beta \left(  \mathbb{N}\right)  \backslash \mathbb{N}$ for every $f\in
C_{R}\left(  \beta \left(  \mathbb{N}\right)  \right)  $.

Suppose $B$ is a linear basis for $\mathbb{R}$ over $\mathbb{Q}$ and $1\in B$.
For every $n\in \mathbb{N}$, let $I_{n}$ be the sequence with $1$ in the $n$-th
place and $0$ other places. For every $n\in \mathbb{N}$, define $T_{n}%
:B\rightarrow \mathbb{C}$ by
\[
T_{n}\left(  b\right)  =T\left(  b\cdot I_{n}\right)  \left(  n\right)  \text{
for every }b\in B\text{.}%
\]

If $k\neq n$, then $T\left(  b\cdot I_{n}\right)  \left(  k\right)
=h_{b}\left(  k\right)  \cdot b\cdot I_{n}\left(  k\right)  =0$. Thus we have
$T_{n}\left(  b\right)  \cdot I_{n}=T\left(  b\cdot I_{n}\right)  $.

We have known $T_{n}\left(  1\right)  =0$ for every $n\in \mathbb{N}$. Next we
prove there is a $n_{T}\in \mathbb{N}$ such that. $T_{n}=0$ for every $n\geq
n_{T}$.

Else, we can get a subsequence $\left \{  b_{n_{k}}\right \}  \subseteq B$ such
that. $T_{n_{k}}\left(  b_{n_{k}}\right)  \neq0$ for every $k\in \mathbb{N}$.
For every $k\in \mathbb{N}$, choose an $q_{k}\in \mathbb{Q}$ such that.
$q_{k}>\frac{1}{|T_{n_{k}}\left(  b_{n_{k}}\right)  |}$. We know $q_{k}\cdot
b_{n_{k}}\in \mathbb{R}$ and choose an $r_{k}\in \mathbb{Q}$ such that.
$|q_{k}\cdot b_{n_{k}}-r_{k}|<1$. Define a sequence
\[
a\left(  n\right)  =q_{k}\cdot b_{n_{k}}-r_{k}\text{ when }n=n_{k}\text{, and
}a\left(  n\right)  =0\text{ when }n\neq n_{k}\text{.}%
\]

for every $n\in \mathbb{N}$. Thus $\left \{  a\left(  n\right)  \right \}  $ is a
real bounded sequence.

For every $k\in \mathbb{N}$,%
\begin{align*}
T\left(  a\right)   &  =T\left(  a-a\left(  n_{k}\right)  \cdot I_{n_{k}%
}\right)  +T\left(  a\left(  n_{k}\right)  \cdot I_{n_{k}}\right) \\
T\left(  a\right)  \left(  n_{k}\right)   &  =0+T\left(  a\left(
n_{k}\right)  \cdot I_{n_{k}}\right)  \left(  n_{k}\right) \\
&  =T\left(  \left(  q_{k}\cdot b_{n_{k}}-r_{k}\right)  \cdot I_{n_{k}%
}\right)  \left(  n_{k}\right) \\
&  =q_{k}\cdot T\left(  b_{n_{k}}\cdot I_{n_{k}}\right)  \left(  n_{k}\right)
\\
&  =q_{k}\cdot T_{n_{k}}\left(  b_{n_{k}}\right)  \text{.}%
\end{align*}

since $q_{k}$, $r_{k}\in \mathbb{Q}$. But $|q_{k}\cdot T_{n_{k}}\left(
b_{n_{k}}\right)  |>1$ for every $k\in \mathbb{N}$. Thus $T\left(  a\right)
\notin C_{0}\left(  \mathbb{N}\right)  $. This is a contradiction.

Thus there is a $n_{T}\in \mathbb{N}$ such that. $T_{n}=0$ for every $n\geq
n_{T}$. Now for any bounded $f\in C_{R}\left(  \mathbb{N}\right)  $,%
\begin{align*}
T\left(  f\right)  \left(  n\right)   &  =T\left(  f-f\left(  n\right)  \cdot
I_{n}\right)  \left(  n\right)  +T\left(  f\left(  n\right)  \cdot
I_{n}\right)  \left(  n\right) \\
&  =0+0=0
\end{align*}

for every $n\geq n_{T}$. Let $D=\left \{  n\in \mathbb{N}\text{: }1\leq n\leq
n_{T}\right \}  $. Thus $T\left(  f\right)  =0$ on $\beta \left(  \mathbb{N}%
\right)  \backslash D$ for every $f\in C_{R}\left(  \beta \left(
\mathbb{N}\right)  \right)  $.

Then $T\left(  f\right)  =T\left(  1\right)  \cdot f$ on $\beta \left(
\mathbb{N}\right)  \backslash D$ for every $f\in C_{R}\left(  \beta \left(
\mathbb{N}\right)  \right)  $.
\end{proof}

\begin{remark}
Note that the $n_{T}$ is in terms of $T$, so we cannot say that there is a
finite subset $D$ of $\mathbb{N}$ such that for every additive local
multiplication $T$ on $C\left(  \beta \left(  \mathbb{N}\right)  \right)  $,
$T\left(  f\right)  =T\left(  1\right)  \cdot f$ on $\beta \left(
\mathbb{N}\right)  \backslash D$ for every $f\in C_{R}\left(  \beta \left(
\mathbb{N}\right)  \right)  $. In fact, we can get that for every additive
local multiplication $T$ on $C\left(  \beta \left(  \mathbb{N}\right)  \right)
$, $T\left(  f\right)  =T\left(  1\right)  \cdot f$ on $\beta \left(
\mathbb{N}\right)  \backslash \mathbb{N}$ for every $f\in C_{R}\left(
\beta \left(  \mathbb{N}\right)  \right)  $.
\end{remark}

\section{$\upsilon$-Spaces}

\begin{definition}
Suppose $X$ is a compact Hausdorff space. Call a map $T$ on $C\left(
X\right)  $ ($C_{R}\left(  X\right)  $) \textbf{zero-preserving} if for every
$f\in C\left(  X\right)  $($C_{R}\left(  X\right)  $) and every $x\in X$,
$f\left(  x\right)  =0$ implies $T\left(  f\right)  \left(  x\right)  =0$.
\end{definition}

The zero-preserving property is weaker than being a local multiplication.
Every local multiplication has this property. However, in general, not every
zero-preserving additive map is a local multiplication. For example, any map
$T$ of the form
\[
T\left(  f\right)  =T\left(  1\right)  \operatorname{Re}\left(  f\right)
+T\left(  i\right)  \operatorname{Im}\left(  f\right)  \text{.}%
\]

For every additive $T$ on $C\left(  X\right)  $ there corresponds four
additive maps on $C_{R}\left(  X\right)  $, namely,
\begin{align*}
T_{1}\left(  u\right)   &  =\operatorname{Re}\left(  T(u)\right) \\
T_{2}\left(  u\right)   &  =\operatorname{Im}\left(  T(u)\right) \\
T_{3}\left(  u\right)   &  =\operatorname{Re}\left(  T\left(  iu\right)
\right) \\
T_{4}\left(  u\right)   &  =\operatorname{Im}\left(  T\left(  iu\right)
\right)  \text{ for every }u\in C_{R}\left(  X\right)  \text{,}%
\end{align*}
and it turns out that $T$ is zero-preserving if and only if each of these four
maps is zero-preserving and $T$ has the form
\[
T\left(  f\right)  =T_{1}\left(  \operatorname{Re}\left(  f\right)  \right)
+iT_{2}\left(  \operatorname{Re}\left(  f\right)  \right)  +T_{3}\left(
\operatorname{Im}\left(  f\right)  \right)  +iT_{4}\left(  \operatorname{Im}%
\left(  f\right)  \right)  \text{.}%
\]

\begin{definition}
We call a compact Hausdorff space $X$ a $\mathbf{\upsilon}$\textbf{-space} if
every zero-preserving additive map on $C_{R}\left(  X\right)  $ is a
multiplication. Equivalently $X$ is a $\upsilon$-space if and only if every
zero-preserving map on $C\left(  X\right)  $ has the form%
\[
T\left(  f\right)  =T\left(  1\right)  \operatorname{Re}\left(  f\right)
+T\left(  i\right)  \operatorname{Im}\left(  f\right)  \text{.}%
\]

\end{definition}

The relationship between local multiplications and zero-preserving maps is
given in the following lemma.

\begin{lemma}
Suppose $T$ is any map on $C\left(  X\right)  $. Then $T$ is a local
multiplication if and only if $T$ leaves invariant every ideal of $C\left(
X\right)  $ and $T$ is zero-preserving if and only if $T$ leaves invariant
every closed ideal of $C\left(  X\right)  $.
\end{lemma}

\begin{proof}
Suppose $T$ is a local multiplication and $I$ is an ideal of $C\left(
X\right)  $. Then for every $f\in I$, $T\left(  f\right)  =h_{f}\cdot f\in I$.
Now suppose $T$ leaves invariant every ideal of $C\left(  X\right)  $. Suppose
$f\in C\left(  X\right)  $ and $I_{f}$ is an ideal generated by $f$. Since
$T\left(  f\right)  \in I_{f}$, $T\left(  f\right)  =h\cdot f$ for some $h\in
C\left(  X\right)  $. $T$ is a local multiplication.

Suppose $T$ is zero-preserving and $I$ is a closed ideal of $C\left(
X\right)  $. Then there is a closed subset $K$ of $X$ such that $I$ is the set
of all functions in $C\left(  X\right)  $ which vanish on $K$. For every $f\in
I$, $f=0$ on $K$. Thus $T\left(  f\right)  =0$ on $K$, i.e., $T\left(
f\right)  \in I$. Now suppose $T$ leaves invariant every closed ideal of
$C\left(  X\right)  $. For every $f\in C\left(  X\right)  $ and every $x\in
X$, if $f\left(  x\right)  =0$, let $I_{x}$ be the ideal of all functions in
$C\left(  X\right)  $ which vanish on $\left \{  x\right \}  $. Thus $I_{x}$ is
closed. Thus $T\left(  f\right)  \in I_{x}$, $T\left(  f\right)  \left(
x\right)  =0$, i.e., $T$ is zero-preserving.
\end{proof}

Here we prove a purely algebraic result that relates to zero-preserving maps,
since the evaluation maps at points of $X$ are algebra homomorphisms of
$C\left(  X\right)  $ to $\mathbb{C}$ and $C_{R}\left(  X\right)  $ into
$\mathbb{R}$, respectively.

\begin{theorem}
\label{zpAlg}Suppose $\mathcal{A}$ is an algebra with identity $1$ over a
field $\mathbb{F}$. Suppose also that, whenever $x\in \mathcal{A}$ and $x\neq
0$, there is an algebra homomorphism $h$ from $\mathcal{A}$ to $\mathbb{F}$
such that $h\left(  x\right)  \neq0$. Let $S$ be the set of unital algebra
homomorphism from $\mathcal{A}$ to $\mathbb{F}$. Suppose $T$ is a linear map
from $\mathcal{A}$ to $\mathcal{A}$ such that, for every $a$ in $\mathcal{A}$
and every $s$ in $S$, $s\left(  a\right)  =0$ implies $s\left(  T\left(
a\right)  \right)  =0$. Then $T$ is left multiplication by $T\left(  1\right)
$.
\end{theorem}

\begin{proof}
If there exists an $a\in \mathcal{A}$ such that. $T\left(  a\right)  \neq
T\left(  1\right)  a$, i.e., $T\left(  a\right)  -T\left(  1\right)  a\neq0$.
From the assumption, there is an algebra homomorphism $h$ from $\mathcal{A}$
to $\mathbb{F}$ such that. $h\left(  T\left(  a\right)  -T\left(  1\right)
a\right)  \neq0$. But%
\begin{align*}
h\left(  T\left(  a\right)  -T\left(  1\right)  a\right)   &  =h\left(
T\left(  a\right)  \right)  -h\left(  T\left(  1\right)  \right)  \cdot
h\left(  a\right) \\
&  =h\left(  T\left(  a\right)  \right)  -h\left(  h\left(  a\right)  \cdot
T\left(  1\right)  \right) \\
&  =h\left(  T\left(  a\right)  -h\left(  a\right)  \cdot T\left(  1\right)
\right) \\
&  =h\left(  T\left(  a\right)  -T\left(  h\left(  a\right)  \cdot1\right)
\right) \\
&  =h\left(  T\left(  a-h\left(  a\right)  \cdot1\right)  \right)  \text{.}%
\end{align*}

Since $h\in S$ and
\[
h\left(  a-h\left(  a\right)  \cdot1\right)  =h\left(  a\right)  -h\left(
a\right)  \cdot h\left(  1\right)  =0\text{,}%
\]
we have
\[
h\left(  T\left(  a-h\left(  a\right)  \cdot1\right)  \right)  =0.
\]
This is a contradiction.
\end{proof}

\begin{corollary}
\label{linear}Suppose $X$ is a compact Hausdorff space. Then any linear map on
$C\left(  X\right)  $ ($\mathbb{R}$-linear map on $C_{R}\left(  X\right)  $)
that is zero-preserving is a multiplication.
\end{corollary}

\begin{corollary}
Suppose $X$ is a compact Hausdorff space. Then $T$ is an $\mathbb{R}$-linear
zero-preserving map on $C\left(  X\right)  $\bigskip \ if and only if, $T$ has
the form%
\[
T\left(  f\right)  =T\left(  1\right)  \operatorname{Re}\left(  f\right)
+T\left(  i\right)  \operatorname{Im}\left(  f\right)  \text{.}%
\]

\end{corollary}

The equivalence of $\left(  1\right)  $ and $\left(  5\right)  $ in the
following theorem, shows that if $X$ is an $\upsilon$-space, then $X$ is like
a "real" $\eta$-space.

\begin{theorem}
\label{uCharacterization}Suppose $X$ is a compact Hausdorff space. The
following are equivalent:

\begin{enumerate}
\item $X$ is a $\upsilon$-space.

\item Every additive zero-preserving map on $C\left(  X\right)  $ is
$\mathbb{R}$-linear.

\item Every additive zero-preserving map on $C\left(  X\right)  $ is continuous.

\item Every additive zero-preserving map on $C\left(  X\right)  $ has the form%
\[
T\left(  f\right)  =T\left(  1\right)  \operatorname{Re}\left(  f\right)
+T\left(  i\right)  \operatorname{Im}\left(  f\right)  \text{ for every }f\in
C\left(  X\right)  \text{.}%
\]

\item Every additive zero-preserving map on $C_{R}\left(  X\right)  $ is a multiplication.
\end{enumerate}
\end{theorem}

\begin{proof}
$\left(  5\right)  \Rightarrow \left(  4\right)  $: We have already known that
for every additive $T$ on $C\left(  X\right)  $ there corresponds four
additive maps on $C_{R}\left(  X\right)  $, namely,
\begin{align*}
T_{1}\left(  u\right)   &  =\operatorname{Re}\left(  T(u)\right) \\
T_{2}\left(  u\right)   &  =\operatorname{Im}\left(  T(u)\right) \\
T_{3}\left(  u\right)   &  =\operatorname{Re}\left(  T\left(  iu\right)
\right) \\
T_{4}\left(  u\right)   &  =\operatorname{Im}\left(  T\left(  iu\right)
\right)  \text{ for every }u\in C_{R}\left(  X\right)  \text{,}%
\end{align*}
and it turns out that $T$ is zero-preserving if and only if each of these four
maps is zero-preserving and $T$ has the form
\[
T\left(  f\right)  =T_{1}\left(  \operatorname{Re}\left(  f\right)  \right)
+iT_{2}\left(  \operatorname{Re}\left(  f\right)  \right)  +T_{3}\left(
\operatorname{Im}\left(  f\right)  \right)  +iT_{4}\left(  \operatorname{Im}%
\left(  f\right)  \right)  \text{.}%
\]

Since every additive zero-preserving map on $C_{R}\left(  X\right)  $ is a
multiplication. We have%
\begin{align*}
T\left(  f\right)   &  =T_{1}\left(  1\right)  \operatorname{Re}\left(
f\right)  +iT_{2}\left(  1\right)  \operatorname{Re}\left(  f\right)
+T_{3}\left(  1\right)  \operatorname{Im}\left(  f\right)  +iT_{4}\left(
1\right)  \operatorname{Im}\left(  f\right) \\
&  =T\left(  1\right)  \operatorname{Re}\left(  f\right)  +T\left(  i\right)
\operatorname{Im}\left(  f\right)  \text{.}%
\end{align*}

$\left(  4\right)  \Rightarrow \left(  3\right)  $: Suppose $\left \{
f_{n}\right \}  $ is a sequence in $C\left(  X\right)  $ and $f_{n}\rightarrow
f$. Then $T\left(  f_{n}\right)  =T\left(  1\right)  \operatorname{Re}\left(
f_{n}\right)  +T\left(  i\right)  \operatorname{Im}\left(  f_{n}\right)
\rightarrow T\left(  1\right)  \operatorname{Re}\left(  f\right)  +T\left(
i\right)  \operatorname{Im}\left(  f\right)  =T\left(  f\right)  $. Thus $T$
is continuous.

$\left(  3\right)  \Rightarrow \left(  2\right)  $: Suppose $a\in \mathbb{R}$
and $\left \{  r_{n}\right \}  $ is a rational sequence which converges to $a$.
Then $r_{n}\cdot f\rightarrow a\cdot f$. Thus%
\[
T\left(  a\cdot f\right)  =\lim T\left(  r_{n}\cdot f\right)  =\lim
r_{n}T\left(  f\right)  =aT\left(  f\right)  \text{.}%
\]

Thus $T$ is $\mathbb{R}$-linear.

$\left(  2\right)  \Rightarrow \left(  1\right)  $: Suppose $T$ is an additive
zero-preserving map on $C_{R}\left(  X\right)  $. If $T$ is not $\mathbb{R}%
$-linear. Define $T_{0}$ on $C\left(  X\right)  $ by $T_{0}\left(  f\right)
=T\left(  \operatorname{Re}\left(  f\right)  \right)  +iT\left(
\operatorname{Im}\left(  f\right)  \right)  $. Thus $T_{0}$ is an additive
zero-preserving map on $C\left(  X\right)  $ which is not $\mathbb{R}$-linear.
This is a contradiction. Then $T$ is $\mathbb{R}$-linear. From Corollary
\ref{linear}, $T$ is a multiplication. Thus $X$ is a $\upsilon$-space.

$\left(  1\right)  \Rightarrow \left(  5\right)  $: This follows immediately
from the definition of $\upsilon$-space.
\end{proof}

\begin{example}
Suppose $X=\left \{  a\right \}  $. Then $C_{R}\left(  X\right)  $ is isomorphic
to $\mathbb{R}$. Every additive map on $\mathbb{R}$ is zero-preserving. The
number of additive (i.e., $\mathbb{Q}$-linear) maps on $\mathcal{\mathbb{R}}$
is $2^{2^{\aleph_{0}}}$ but the number of multiplications is $2^{\aleph_{0}}$,
so $X$ is not a $\upsilon$-space.
\end{example}

As in the $\eta$-space case, we get the following for free.

\begin{theorem}
Suppose $X$ is the disjoint union of compact Hausdorff spaces $Y$ and $Z$.
Then $X$ is a $\upsilon$-space if and only if $Y$ and $Z$ are $\upsilon$-spaces.
\end{theorem}

\begin{corollary}
A $\upsilon$-space has no isolated points.
\end{corollary}

\bigskip

If we examine the proofs of Theorems in the preceding section, we immediately
obtain the following results.

\begin{theorem}
\label{fromsec2}Suppose $X$ is a compact Hausdorff space. Then

\begin{enumerate}
\item If the set of sequential limit points is dense in $X$, then $X$ is a
$\upsilon$-space.

\item If $X$ is first countable, then $X$ is a $\upsilon$-space if and only if
$X$ has no isolated points.

\item If $\left \{  K_{i}:i\in I\right \}  $ is a collection of closed subsets
of $X$ and if each $K_{i}$ is a $\upsilon$-space, then the closure of the
union of $K_{i}$'s is a $\upsilon$-space.

\item $X$ has a unique maximal compact $\nu$-subspace.

\item If $Y$ is a compact Hausdorff $\upsilon$-space, then $X\times Y$ is a
$\upsilon$-space.
\end{enumerate}
\end{theorem}

\section{P-points and $\mathfrak{q}$-points}

We now want to generalize Theorem \ref{slp}. We call a point $x\in X$ a
$q$\textbf{-point} if and only if there is a disjoint sequence $\left \{
K_{n}\right \}  $ of compact subsets of $X$ such that

\begin{enumerate}
\item Each $K_{n}$ is disjoint from the closure of $\cup_{m\in \mathbb{N},m\neq
n}K_{m}$;

\item $x\in \left(  \cup_{n\in \mathbb{N}}K_{n}\right)  ^{-}\backslash \left(
\cup_{n\in \mathbb{N}}K_{n}\right)  $.
\end{enumerate}

We say that $x$ is a \textbf{strong }$\mathfrak{q}$\textbf{-point} if there is
a disjoint sequence $\left \{  K_{n}\right \}  $ of compact sets satisfying

\begin{enumerate}
\item[3.] $x\in \left(  \cup_{n\in \mathbb{N}}K_{2n}\right)  ^{-}\cap \left(
\cup_{n\in \mathbb{N}}K_{2n-1}\right)  ^{-}$.
\end{enumerate}

It is clear that these conditions on the sequence $\left \{  K_{n}\right \}  $
is precisely what is needed to ensure that if $\left \{  z_{n}\right \}  $ is a
sequence of complex numbers converging to $z$, then the function $f$ on
$\cup_{n=1}^{\infty}K_{n}$ defined by $f|_{K_{n}}=z_{n}$ extends to $\left(
\cup_{n\in \mathbb{N}}K_{2n}\right)  ^{-}$, which by the Tietze extension
theorem extends to a continuous function on $X$. By examining the proof of
Theorem \ref{slp}, we easily obtain the following.

\begin{theorem}
\label{qp}Suppose $X$ is a compact Hausdorff space.

\begin{enumerate}
\item If the set of $\mathfrak{q}$-points is dense in $X$, then every additive
local multiplication $T$ on $X$ has the form $T\left(  f\right)  =T\left(
1\right)  \operatorname{Re}\left(  f\right)  +T\left(  i\right)
\operatorname{Im}\left(  f\right)  $ for some $T\left(  1\right)  $,$T\left(
i\right)  \in C\left(  X\right)  $.

\item If the set of strong $\mathfrak{q}$-points is dense in $X$, then $X$ is
an $\eta$-space.
\end{enumerate}
\end{theorem}

\bigskip

The notion of a $\mathfrak{q}$-point is very closely related to the classical
notion of a P-point defined by L. Gillman and M. Henriksen \cite{GH}, which
can be defined as follows:

$x\in X$ is a \textbf{P-point} in $X$ if and only if every continuous function
in $C\left(  X\right)  $ is constant on some open set containing $x$.

\begin{lemma}
\label{pq}Suppose $X$ is a compact Hausdorff space. Then

\begin{enumerate}
\item If $x\in X$, then $x$ is a $\mathfrak{q}$-point if and only if $x$ is
not a P-point.

\item If $K\subseteq X$ is compact and every point of $K$ is a P-point of $X$,
then $K$ is finite.

\item If $A$ is the set of all $\mathfrak{q}$-points of $X$, then
$X\backslash \bar{A}$ is the set of isolated points of $X$.
\end{enumerate}
\end{lemma}

\begin{proof}
$\left(  1\right)  $. Suppose $x$ is a $\mathfrak{q}$-point. We can choose a
disjoint collection $\left \{  K_{1},K_{2},\ldots \right \}  $ of compact sets
such that each is disjoint from the union of the others, $x\in \left(
\cup_{n=1}^{\infty}K_{n}\right)  ^{-}\backslash \left(  \cup_{n=1}^{\infty
}K_{n}\right)  $. We can define a continuous function $f:\left(  \cup
_{n=1}^{\infty}K_{n}\right)  ^{-}\rightarrow \left[  0,1\right]  $ so that
$f\left(  a\right)  =1/n$ when $a\in K_{n}$ and $f\left(  a\right)  =0$ on
$\left(  \cup_{n=1}^{\infty}K_{n}\right)  ^{-}\backslash \left(  \cup
_{n=1}^{\infty}K_{n}\right)  $. By the Tietze extension theorem, we can assume
$f\in C\left(  X\right)  $. It is clear that $f\left(  x\right)  =0$ but there
is no neighborhood of $x$ on which $f$ is $0$. Thus $x$ is not a P-point.

Conversely, suppose $x$ is not a P-point and suppose $g\in C\left(  X\right)
$ and $g\left(  x\right)  =0$ but $g$ is not $0$ on a neighborhood of $x$.
Thus
\[
x\in \left[  X\backslash Z\left(  g\right)  \right]  ^{-}.
\]
By replacing $g$ with $\left \vert g\right \vert /\left(  1+\left \vert
g\right \vert \right)  $ we can assume that $0\leq g\leq1$. For each
$n\in \mathbb{N}$, let $E_{n}=\left \{  x\in X:\frac{1}{n+1}\leq g\left(
x\right)  \leq \frac{1}{n}\right \}  $. We know that $x$ is in the closure of
the union of the $E_{n}$'s. Thus $x$ is either in the union of the closure of
$\cup_{n\in \mathbb{N}}E_{2n}$ or the closure of $\cup_{n\in \mathbb{N}}%
E_{2n-1}$. In the former case we let $K_{n}=E_{2n}$ for each $n$, and in the
latter case we let $K_{n}=E_{2n-1}$ for each $n$. In either case we see that
$x$ is a $\mathfrak{q}$-point.

$\left(  2\right)  .$ Suppose $K\subseteq X$ is compact and every point of $K$
is a P-point of $X.$ It follows from the Tietze extension theorem that every
point of $K$ is a P-point of $K$. It follows from Proposition 4.1 in \cite{M}
that $K$ is finite.

$\left(  3\right)  $. Let $E$ be the closure of the set of all $\mathfrak{q}%
$-points of $X$. Suppose $x_{0}\in X\backslash E$. By Urysohn's lemma we can
find a continuous $h:X\rightarrow \left[  0,1\right]  $ such that $h\left(
x_{0}\right)  =1$ and $h|_{E}=0$. Thus $\left \{  x\in X:h\left(  x\right)
\geq1/2\right \}  $ is a compact subset of $X$ for which every point is a
P-point of $X$, which means it is finite. Thus $\left \{  x\in X:g\left(
x\right)  >1/2\right \}  $ is a finite open set containing $x_{0}$. Hence
$x_{0}$ is an isolated point of $X$. Also no isolated point is in $E$, so
$X\backslash E$ is the set of isolated points of $X$.
\end{proof}

\bigskip

\section{Main Results}

We can now give complete characterizations of $\eta$-spaces, real $\eta
$-spaces and $\upsilon$-spaces. Here is our first main theorem, which shows
that being a real $\eta$-space and being a $\upsilon$-space are the same as
having no isolated points.

\begin{theorem}
\label{main}Suppose $X$ is a compact Hausdorff space. Then the following are equivalent:

\begin{enumerate}
\item $X$ is a $\upsilon$-space.

\item $X$ is a real $\eta$-space.

\item The set of $\mathfrak{q}$-points of $X$ is dense in $X$.

\item $X$ has no isolated points.
\end{enumerate}
\end{theorem}

\begin{proof}
We already proved that $\left(  3\right)  \Rightarrow \left(  1\right)
\Rightarrow \left(  2\right)  \Rightarrow \left(  4\right)  $. The proof of
$\left(  4\right)  \Rightarrow \left(  3\right)  $ follows from part $\left(
3\right)  $ of Lemma \ref{pq}.
\end{proof}

\bigskip

The following result characterizes $\eta$-spaces.

\begin{theorem}
\label{main2}Suppose $X$ is a compact Hausdorff space. The the following are equivalent:

\begin{enumerate}
\item $X$ is an $\eta$-space.

\item No nonempty open F$_{\sigma}$ set in $X$ is an F-space.

\item $X$ has no isolated points and, for every $0\neq g\in C\left(  X\right)
$, the map $T\left(  f\right)  =g\bar{f}$ is \textbf{not} a local multiplication.
\end{enumerate}
\end{theorem}

\begin{proof}
$\left(  1\right)  \Rightarrow \left(  2\right)  $: If $X$ is an $\eta$-space,
it follows from Theorem \ref{gfbar} that no nonempty open F$_{\sigma}$ set in
$X$ is an F-space.

$\left(  2\right)  \Rightarrow \left(  3\right)  $: Suppose that no nonempty
open F$_{\sigma}$ set in $X$ is an F-space. It follows that $X$ has no
isolated points. Also, by Lemma \ref{oldgfbar}, for every $0\neq g\in C\left(
X\right)  $, the map $T\left(  f\right)  =g\bar{f}$ is not a local multiplication.

$\left(  3\right)  \Rightarrow \left(  1\right)  $: By Theorem \ref{main}, $X$
is an $\upsilon$-space. Suppose $T$ is an additive local multiplication on
$X.$ Since $X$ is an $\upsilon$-space, $T$ must have the form%
\[
T\left(  f\right)  =T\left(  1\right)  \operatorname{Re}\left(  f\right)
+T\left(  i\right)  \operatorname{Im}\left(  f\right)  .
\]
Thus $T$ is $\mathbb{R}$-linear. It follows from Corollary \ref{rl} that $T$
is a multiplication.
\end{proof}

\bigskip

We now describe how the maximal subsets of a compact Hausdorff space $X$ that
are $\eta$-spaces or $\upsilon$-spaces can be constructed.

\begin{remark}
\label{transfinite}One might guess that the maximal $\upsilon$-subspace of $X$
is the closure of the set of $\mathfrak{q}$-points of $X$. However, if
$X=\left \{  0,1,1/2,1/3,1/4,\ldots \right \}  $ then the closure of the set of
$\mathfrak{q}$-points of $X$ is precisely $\left \{  0\right \}  $. However,
this is not a $\mathfrak{u}$-space. We have to use a transfinite construction
argument. If $K$ is a compact Hausdorff space we define
\[
\mathfrak{n}\left(  K\right)  =K\backslash \left \{  x\in K:x\text{ is an
isolated point}\right \}
\]%
\[
=\left \{  x\in K:x\text{ is a }\mathfrak{q}\text{-point of }K\right \}
^{-}\text{.}%
\]
We let $E_{0}=X$, and suppose $\alpha>0$ is an ordinal such that, for all
$\beta>\alpha$, $E_{\beta}$ is defined. We define $E_{\alpha}$ by%
\[
E_{\alpha}=\left \{
\begin{array}
[c]{cc}%
\mathfrak{n}\left(  E_{\beta}\right)  & \text{if }\alpha=\beta+1\\
\cap_{\beta<\alpha}E_{\beta} & \text{if }\alpha \text{ is a limit ordinal}%
\end{array}
\right.  .
\]
Since $\left \{  E_{\beta}\backslash E_{\beta+1}:\beta \text{ is an
ordinal}\right \}  $ is a disjoint collection of subsets of $X$, there is a
smallest ordinal $\alpha$ such that $E_{\alpha}=E_{\alpha+1}$. It is clear
that if $K$ is a compact subset of $X$ having no isolated points, then
$K\subset E_{\beta}$ for every ordinal $\beta$; in particular, $K\subset
E_{\alpha}$. Since $E_{\alpha}=E_{\alpha+1}=\mathfrak{n}\left(  E_{\alpha
}\right)  $, $E_{\alpha}$ has no isolated points. Thus $E_{\alpha}$ is the
maximal compact subset of $X$ that has no isolated points, i.e., the maximal
compact subset of $X$ that is a $\upsilon$-space.

The maximal $\eta$-subspace is a little more complicated. First suppose $Y$ is
a compact subset of $X$ and $Y$ is an $\eta$-space. Suppose also that $V$ is
an open $F_{\sigma}$ subset of $X$ and that $V$ is an F-space. It follows from
Lemma \ref{oldgfbar} that there is a $g\in C\left(  X\right)  $ such that
$X\backslash Z\left(  g\right)  =V$ and the map $T\left(  f\right)  =g\bar{f}$
is a local multiplication on $C\left(  X\right)  $. If $g_{Y}=g|_{Y}\neq0$,
then the map $S\left(  f\right)  =g_{Y}\bar{f}$ is a local multiplication on
$C\left(  Y\right)  $ that is not a multiplication. Since $Y$ is an $\eta
$-space, $g|_{Y}=0,$ or $Y\cap V=\varnothing$.

We now imitate the process above. If $K$ is a compact Hausdorff space, define%
\[
\mathfrak{m}\left(  K\right)  =K\backslash \cup \left \{  V\subset X:V\text{ is
an open F}_{\sigma}\text{ and an F-space}\right \}  \text{.}%
\]
We define $F_{0}=X$, and for $\alpha>0$ define%
\[
F_{\alpha}=\left \{
\begin{array}
[c]{cc}%
\mathfrak{m}\left(  F_{\beta}\right)  & \text{if }\alpha=\beta+1\\
\cap_{\beta<\alpha}F_{\beta} & \text{if }\alpha \text{ is a limit ordinal}%
\end{array}
\right.
\]
Then we can choose $\alpha$ to be the smallest ordinal such that $F_{\alpha
}=F_{\alpha+1}$. Then $F_{\alpha}$ is the maximal compact $\eta$-subspace of
$X$.
\end{remark}

We conclude with remarks on the space $\beta \left(  \mathbb{N}\right)
\backslash \mathbb{N}$.

\begin{remark}
Since $\beta \left(  \mathbb{N}\right)  \backslash \mathbb{N}$ is an F-space, we
know that $\beta \left(  \mathbb{N}\right)  \backslash \mathbb{N}$ is not an
$\eta$-space. However, $\beta \left(  \mathbb{N}\right)  \backslash \mathbb{N}$
has no isolated points. Thus, by Theorem \ref{main}, $X$ is an $\upsilon
$-space. Thus every local multiplication on $C_{R}\left(  X\right)  $ is a
multiplication, but the same is not true for $C\left(  X\right)  $.
\end{remark}

\begin{remark}
Another interesting fact is that the question of whether every point in
$\beta \left(  \mathbb{N}\right)  \backslash \mathbb{N}$ is a $q$-point is
independent from the axioms of set theory (ZFC). W. Rudin \cite{R} proved that
if you assume the continuum hypothesis, then $\beta \left(  \mathbb{N}\right)
\backslash \mathbb{N}$ contains a P-point. Later S. Shelah (see \cite{W})
proved that there is a model of set theory in which every point of
$\beta \left(  \mathbb{N}\right)  /\mathbb{N}$ is a $\mathfrak{q}$-point.
\end{remark}

\end{document}